\documentclass[12pt]{amsart}
\usepackage{graphicx,dbnsymb,amssymb,floatflt,color,epic,eepic,pb-diagram,xfrac}
\usepackage{caption}
\usepackage{subcaption}
\usepackage{picins,paralist} 
\usepackage{amsmath}
\usepackage[all]{xy}

\newtheorem{theorem}{Theorem}
\newtheorem{proposition}{Proposition}[section]

\newtheorem{lemma}[proposition]{Lemma}
\newtheorem{corollary}[proposition]{Corollary}

\theoremstyle{definition}
\newtheorem{definition}[proposition]{Definition}

\newtheorem{example}[proposition]{Example}

\newtheorem{remark}[proposition]{Remark}

\newlength{\standardunitlength}
\catcode`\@=11
\long\def\@makecaption#1#2{%
    \vskip 10pt
    \setbox\@tempboxa\hbox{
      \small\sf{\bfcaptionfont #1. }\ignorespaces #2}%
    \ifdim \wd\@tempboxa >\captionwidth {%
        \rightskip=\@captionmargin\leftskip=\@captionmargin
        \unhbox\@tempboxa\par}%
      \else
        \hbox to\hsize{\hfil\box\@tempboxa\hfil}%
    \fi}
\font\bfcaptionfont=cmssbx10 scaled \magstephalf
\newdimen\@captionmargin\@captionmargin=2\parindent
\newdimen\captionwidth\captionwidth=\hsize
\catcode`\@=12

\def\qed{{\hfill\text{$\Box$}}}

\newlength{\globalparindent}
\def\calA{{\mathcal A}}
\def\calC{{\mathcal C}}

\def\calSo{{\mathcal S}_o}

\def\calD{{\mathcal D}}

\newcommand{\Cob}{{\mathcal Cob}}
\newcommand{\Cobo}{{\mathcal Cob}^3_o}

\newcommand{\Cobdl}{{\mathcal Cob}_{\bullet/l}}

\newcommand{\Kh}{{\text{\it Kh}}}

\newcommand{\Kob}{\operatorname{Kob}}

\newcommand{\Kobo}{\operatorname{Kob}_o}

\newcommand{\Koboh}{{\operatorname{Kob}_{o/h}}}
\newcommand{\Kobh}{{\operatorname{Kob}_{/h}}}

\newcommand{\Kom}{\operatorname{Kom}}
\newcommand{\Komh}{\operatorname{Kom}_{/h}}

\newcommand{\Mat}{\operatorname{Mat}}

\renewcommand{\qed}{~\hfill$\square$}

\begin{document}
\newdimen\captionwidth\captionwidth=\hsize

\title{Khovanov Homology for Alternating Tangles}

\author{Dror Bar-Natan}
\address{
 Department of Mathematics\\
 University of Toronto\\
 Toronto Ontario M5S 2E4\\
 Canada
}
\email{drorbn@math.toronto.edu} \urladdr{http://www.math.toronto.edu/~drorbn}

\author{Hernando Burgos-Soto} \address{
 George Brown College\\
 Toronto Ontario M5A 1N1\\
 Canada
}
\email{hburgos@georgebrown.ca}
\urladdr{http://individual.utoronto.ca/hernandoburgos/}

\date{
  First edition: March.{} 22, 2010.
  This edition: \today
}


\thanks{The first author was partially supported by NSERC grant RGPIN 262178}

\begin{abstract}
  We describe a ``concentration on the diagonal" condition on the Khovanov complex of tangles, show that this condition is satisfied by the Khovanov complex of the single crossing tangles $(\overcrossing)$ and
$(\undercrossing)$, and prove that it is preserved by alternating planar algebra compositions. Hence, this condition is satisfied by the Khovanov complex of all alternating tangles. Finally, in the case of $0$-tangles, meaning links, our condition is equivalent to a well known result \cite{Lee1} which states that the Khovanov homology of a non-split alternating link is supported on two diagonals. Thus our condition is a generalization of Lee's Theorem to the case of tangles
\end{abstract}


\maketitle

\tableofcontents

\section{Introduction} \label{sec:intro}

Khovanov \cite{Khov1} constructed an invariant of links which opened
new prospects in knot theory and which is now known as Khovanov
homology. Bar-Natan in \cite{Bar}  computed
this invariant and found that it is a stronger invariant than the
Jones polynomial. Khovanov, Bar-Natan and Garoufalidis \cite{Ga}
formulated several conjectures related to the Khovanov homology. One
of these refers to the fact that the Khovanov homology of a
non-split alternating link is supported in two lines. To see this,
in Table \ref{Tab:homborr}, we present the dimension of the groups in
the Khovanov homology for the Borromean link and illustrate that the
non-trivial groups are located in two consecutive diagonals.
The fact that every alternating link
satisfies this property was proved by Lee in \cite{Lee1}. \\
\begin{table}
\centering \begin{tabular}{|c||c|c|c|c|c|c|c|}
  \hline
  j$\setminus$i & -3 & -2 & -1 & 0 & 1 & 2 & 3 \\
  \hline \hline
  7 &  &  &  &  &  &  & 1\\
  \hline
  5 &  &  &  &  &  & 2 &  \\
  \hline
  3 &  &  &  &  &  & 1 & \\
  \hline
  1 &  &  &  & 4 & 2 &  & \\
  \hline
  -1 &  &  & 2 & 4 &  &  & \\
  \hline
  -3 &  & 1 &  &  &  &  & \\
  \hline
  -5 &  & 2 &  &  &  &  & \\
  \hline
  -7 & 1 &  &  &  &  &  & \\
  \hline
\end{tabular}
\caption{The Khovanov homology for the Borromean
link}\label{Tab:homborr}
\end{table}
\indent  In \cite{Bar1} Bar-Natan presented a generalization of
Khovanov homology to tangles. In his approach, a formal chain complex is
assigned to every tangle. This formal chain complex, regarded within
a special category, is an (up to homotopy) invariant of the tangle.
 For
the particular case in which the tangle is a link, this chain complex coincides with the cube of smoothings presented in \cite{Khov1}.\\
  \indent This local Khovanov theory was used in \cite{Bar2} to make an algorithm which provides a
  faster computation of the Khovanov homology of a link. The technique used in this paper was also important for theoretical reasons.  We can apply it to prove the invariance of the Khovanov homology, see \cite{Bar2}. It was also used in \cite{Bar3} to give a simple
  proof of Lee's result stated in \cite{Lee2}, about the dimension of the Lee variant of the Khovanov homology.
  Here, we will show how it can be used to state a generalization to tangles of the aforementioned Lee's theorem \cite{Lee1}
about the Khovanov homology of alternating links. Most of the success attained by this algorithm is due to the simplification of the Khovanov complex associated to a tangle. This simplification consists of the elimination of the loops in the smoothing of the complex  (delooping), and the isomorphisms that appear as entries within differentials (Gaussian elimination).  Indeed, given a chain complex $\Omega$ it is possible apply iteratively delooping and gaussian elimination and obtain a homotopy equivalent complex with no loops and no isomorphisms. In this paper, we say that the resulting complex is a \emph{reduced form of  $\Omega$}, and the algorithm that allows to find it will be named the \emph{DG algorithm}. \\
\indent
\parpic[r]{\includegraphics[scale=.65]%
{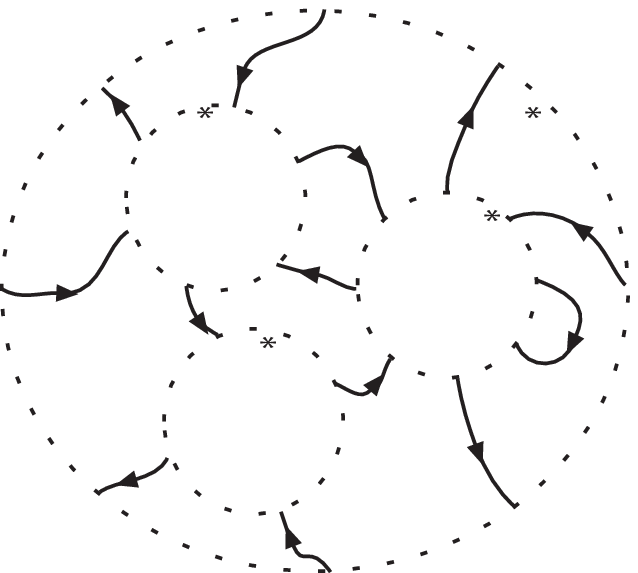}}
 In section \ref{subsec:GravityInformation} we observe that the Khovanov complex of an alternating tangle can be endowed with consistent ``orientations"\footnote{Note that these are orientations of the \emph{smoothings}, and they have nothing to do with the orientations of the components of the tangle itself.}, namely, every strand in every smoothing appearing in the complex can be oriented in a natural way, and likewise every cobordism, in a manner so that these orientations are consistent. (A quick glance at figures \ref{gravicross} on page \pageref{gravicross} and \ref{Fig:smoothing} on page \pageref{Fig:smoothing} should suffice to convince the experts). An important tool for the composition of objects of this type is the concept of \emph{alternating planar algebra}. An alternating planar algebra is an oriented planar algebra as in \cite[Section 5]{Bar1}, where the $d$-input planar arc diagrams $D$ satisfy the following conditions: i) The number $k$ of strings ending on the external boundary of
$D$ is greater than 0. ii) There is complete connection among input discs of the diagram and its arcs, namely, the union of the diagram arcs and the boundary of the internal holes is a connected set. iii) The in- and out-strings alternate in every boundary component of the diagram. A planar arc diagram like this is called a type-$\mathcal{A}$ planar diagram. If $\Phi$ is an element in the planar algebra and $D$ is a $1$-input type-$\mathcal{A}$ planar diagram then $D(\Phi)$ is called a partial closure of $\Phi$.
\indent 
\parpic[r]{\includegraphics[scale=0.9]
{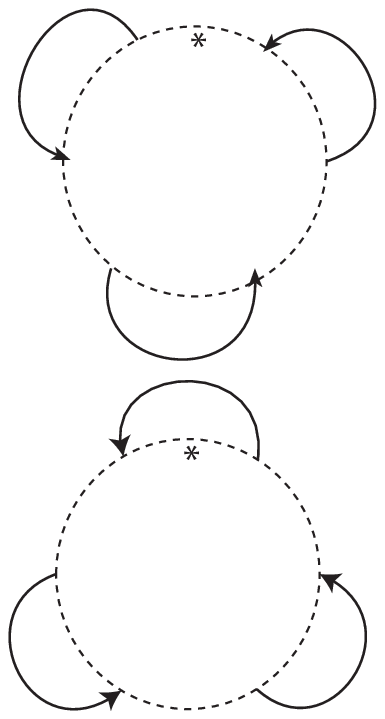}}
Associated with each oriented smoothing $\sigma $ is an integer called the  {\it rotation number} of $\sigma$ which can be determined in the following way. If $\sigma$ has only closed components (loops), then  $R(\sigma)$ is the number of positively oriented (oriented counterclockwise) loops minus the number of negatively oriented (oriented clockwise) loops. The rotation number of an oriented smoothing with boundary is the rotation number of its standard closure, which is obtained by embedding the smoothing in one of the two diagrams on the right, depending on the orientation. Figure \ref{fFig:rotationnumbers} displays several smoothings with their corresponding rotation number. This orientation in the smoothings and the rotation number associated to it were previously utilized in \cite{Bur} to  generalize a Thistlethwaite's result for the Jones polynomial stated in \cite{Thi}. \\

\begin{figure}
        \centering
        \begin{subfigure}[b]{0.3\textwidth}
                \centering
                \includegraphics[scale=1.55]{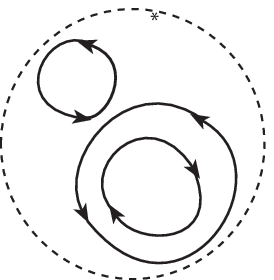}
                \caption{Smoothing with rotation number 1}
                \label{fig:loops}
        \end{subfigure}%
        ~ 
        \begin{subfigure}[b]{0.3\textwidth}
                \centering
                \includegraphics[scale=1.55]{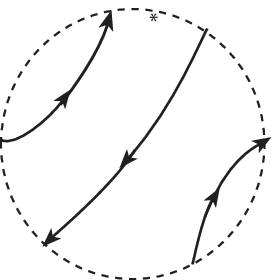}
                \caption{Smoothing with rotation number 2}
                \label{fig:strands}
        \end{subfigure}
        ~ 
        \begin{subfigure}[b]{0.3\textwidth}
                \centering
                \includegraphics[scale=1.55]{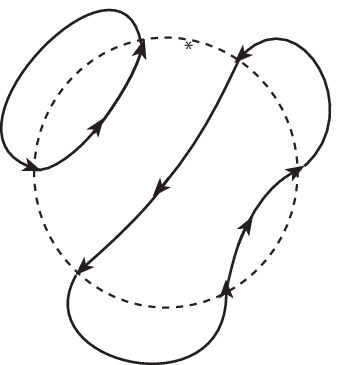}
                \caption{The standard closure of the smoothing in B}
                \label{fig:positiveclosure}
        \end{subfigure}
        \caption{Several smoothings and their rotation numbers. For calculating the rotation number of the smoothing in (B) we count the loops in its positive closure which appears in (C)}\label{fFig:rotationnumbers}
\end{figure}

\indent In a manner similar to \cite{Bar1}, we define a certain graded category $\Cobo$ of {\it oriented
cobordisms}. The objects of $\Cobo$ are {\it oriented smoothing}, and the morphisms are oriented cobordisms. This category is used to define the category  $\Kom(\Mat(\Cobo))$ (abbreviated $\Kobo$) of complexes over $\Mat(\Cobo)$.\\
\indent Specifically, for degree-shifted smoothings
$\sigma\{q\}$, we define $R(\sigma\{q\}): = R(\sigma) + q$. We further use this {\it
degree-shifted rotation number} to define a special class of chain
complexes in $\Kobo$.

\begin{definition}\label{def:diagonal complex} Let $C$ be an integer. A complex $\Omega$ is called \emph{$C$-diagonal} if it is homotopy equivalent to a reduced complex
\[\Omega':\qquad \cdots \longrightarrow \left[\sigma^r_j\right]_j
\longrightarrow \left[\sigma^{r+1}_j\right]_j \longrightarrow \cdots \]
which satisfies that for all degree-shifted smoothings $\sigma^r_j $,  $2r-R(\sigma^r_j)=C$. \end{definition}

\indent In other words, in the reduced form of a diagonal complex, twice the homological
degrees and the degree-shifted rotation numbers of the smoothings always lie along a
single diagonal. The constant $C$ is called
the {\it rotation constant} of $\Omega$. When no confusion arises we only write \emph{diagonal complex} to signify that there exists a constant $C$ such as the complex is $C$-diagonal. Using the above terminology, our main result is stated as follows: 
\begin{theorem}\label{Theorem:MainTheo1} If $T$ is an alternating non-split tangle then there exist a constant $C$ such that the Khovanov homology $Kh(T)$ is $C$-diagonal.
\end{theorem}

Roughly speaking, we say that a complex $\Omega$ is {\it coherently diagonal} (precise meaning is stated in definition \ref{def:cohdiagonalcomplex}) if it is a  diagonal complex whose partial closures 
are also diagonal. As every partial closure of an alternating tangle is again an alternating tangle, Theorem \ref{Theorem:MainTheo1} can be strengthened to say that the Khovanov homology $Kh(T)$ of a non-split alternating tangle is coherently diagonal.  To prove the stronger version of the theorem we use the fact that non-split alternating tangles form an alternating planar algebra  generated by the one-crossing tangles $(\overcrossing)$ and $(\undercrossing)$. Thus Theorem \ref{Theorem:MainTheo1} follows from the observation that $Kh(\overcrossing)$ and $Kh(\undercrossing)$ are coherently diagonal and from Theorem \ref{Theorem:MainTheo2} below:

\begin{theorem}\label{Theorem:MainTheo2} If $\Omega_1,\ldots,\Omega_n$ are coherently diagonal complexes and $D$ is an alternating planar diagram then $D(\Omega_1,\ldots,\Omega_n)$ is coherently diagonal
\end{theorem}
In the case of alternating tangles with no boundary, i.e., in the case of alternating links,
Theorem \ref{Theorem:MainTheo1}
 reduces to Lee's theorem on the Khovanov homology of
alternating links.\\
\indent The work is organized as follows.  Section
\ref{sec:QuickReview} reviews the local Khovanov theory of \cite{Bar1}. Section \ref{sec:Alternating} is devoted to introducing the category $\Cobo$ and gives a quick review of some concepts related to
 alternating planar algebras. In particular we review the concepts of rotation numbers, alternating planar diagrams, associated rotation numbers, and basic operators.\\
\indent Section \ref{sec:On-Diagonal}  presents examples of diagonal and non-diagonal complexes, Introduces the concept of coherently diagonal complexes, and  states some results about the complexes obtained when a basic operator is applied to diagonal complexes. When applied to the composition of two tangles, the Khovanov homology is formed from a double complex. However, when applying the DG algorithm the structure of double complex is lost. Attempting to fix this problem in Section \ref{URC} we introduced the concept of perturbed double complex which is the object in which the DG algorithm lives. Indeed, we shall see that double complexes are special cases of perturbed double complexes and that after applying any step of the DG algorithm in a perturbed double complex, the complex continues being of the same class.  The application of the DG algorithm leads to the proof in section \ref{sec:Theorem2}  of Theorem \ref{Theorem:MainTheo2}. Finally section
  \ref{sec:LeeTheorem} is dedicated to the study of non-split alternating tangles. Here, we prove Theorem \ref{Theorem:MainTheo1} and derive from it the Lee's Theorem formulated in \cite{Lee1}.\\
  
  \section{Acknowledgement} We wish to thank our referee for some helpful comments.
  

\section{The local Khovanov theory: Notation and some details}
\label{sec:QuickReview} \indent The notation and some results
 appearing here are treated in more details in~\cite{Bar1,Bar2,Naot}. Given a set
$B$ of $2k$ marked points on a based circle $C$, a smoothing with
boundary $B$ is a union of strings $a_1,...,a_n$ embedded in the
planar disk for which $C$ is the boundary, such that
$\cup_{i=1}^n\partial a_i=B$. These strings are either closed
curves, {\it loops}, or strings whose boundaries are points on $B$,
{\it
strands}. If $B=\emptyset$, the smoothing is a union of circles.\\
\indent 
\parpic[r]{\includegraphics[scale=1]%
{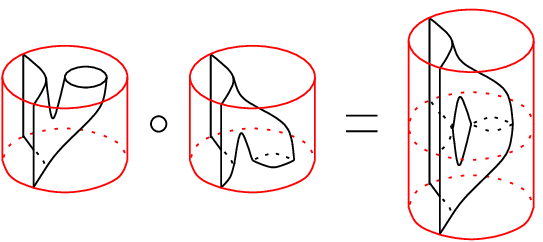}} We denote by $Cob^3(B)$ the category whose objects are
smoothings with boundary $B$, and whose morphisms are cobordisms
between such smoothings, regarded up to boundary preserving isotopy.
The composition
of morphisms is given by placing one cobordism atop the other. \\

As in \cite[Section 11.2]{Bar1}, $\Cobdl^3(B)$ denotes the extension of $Cob^3(B)$ where ``dots" are allowed, and whose morphisms are considered modulo the
local relations:
\begin{equation} \label{eq:LocalRelations}
\begin{array}{c}
  \begin{array}{c}
    \includegraphics[height=1cm]{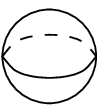}
  \end{array}\hspace{-2mm}=0,
  \qquad\qquad
  \begin{array}{c}
    \includegraphics[height=1cm]{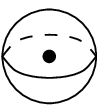}
  \end{array}\hspace{-2mm}=1,
  \qquad\qquad
  \begin{array}{c}\includegraphics[height=10mm]{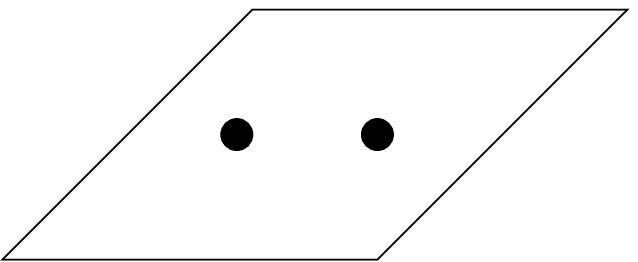}\end{array}
  \hspace{-4mm}=0,
\\
  \text{and}\qquad
  \begin{array}{c}\includegraphics[height=10mm]{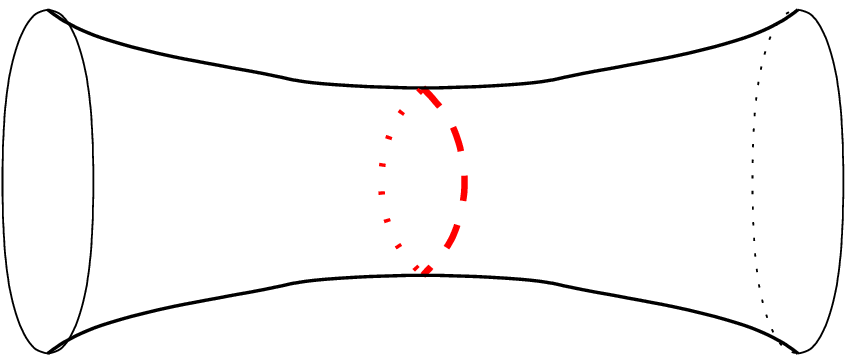}\end{array}
  =\begin{array}{c}\includegraphics[height=10mm]{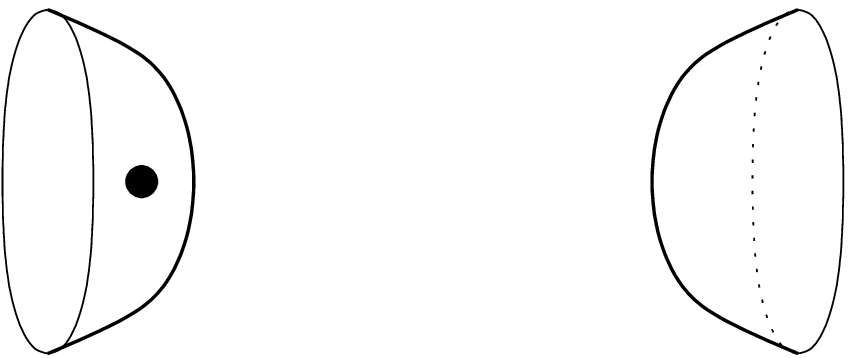}\end{array}
  +\begin{array}{c}\includegraphics[height=10mm]{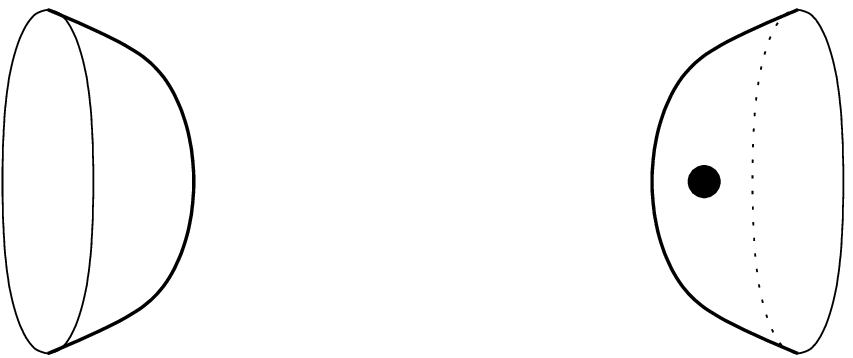}\end{array}.
\end{array}
\end{equation}
We will use the notation $\Cob^3$ and $\Cobdl^3$ as a generic
reference, namely, $\Cob^3=\bigcup_B\Cob^3(B)$ and
$\Cobdl^3=\bigcup_B\Cobdl^3(B)$. If $B$ has $2k$ elements, we usually
write  $\Cobdl^3(k)$ instead of $\Cobdl^3(B)$. If $\mathcal{C}$ is
any category, $\operatorname{Mat}(\mathcal{C})$
 will be the additive category whose objects are column vectors (formal direct sums) whose entries are objects of $\mathcal{C}$. Given two objects
in this category,
\begin{align*} {\mathcal O}=
 \left( \begin{tabular}{c}
  ${\mathcal O}_{1}$ \\ ${\mathcal O}_{2}$\\ \vdots \\ ${\mathcal O}_{n}$  \end{tabular} \right) \hspace{1cm} &{\mathcal O}^1= \left( \begin{tabular}{c}
  ${\mathcal O}_{1}^1$ \\ ${\mathcal O}_{2}^1$ \\ \vdots \\ ${\mathcal O}_{m}^1$  \end{tabular}
  \right),\end{align*}
  the morphisms between these objects will be matrices whose entries are formal sums of morphisms between them.
   The morphisms in this additive category are added using the
  usual matrix addition and the morphism composition is modelled by matrix
  multiplication, i.e, given two appropriate morphisms $F=(F_{ik})$ and $G=(G_{kj})$ between objects of
this category, then $F\circ G$ is given by
\[
F\circ  G= \sum_k F_{ik} G_{kj}\, .
\]

  $\operatorname{Kom}(\mathcal{C})$ will be the category of formal
  complexes
 over an additive category $\mathcal{C}$. $\operatorname{Kom}_{/h}(\mathcal{C})$ is
 $\operatorname{Kom}(\mathcal{C})$ modulo homotopy. We also use the abbreviations
 $\Kob(k)$ and $\Kobh(k)$ for denoting $\Kom(\Mat(\Cobdl^3(k)))$ and $\Komh(\Mat(\Cobdl^3(k)))$.\\
 \indent
Objects and morphisms of the categories $\Cob^3$, $\Cobdl^3$,
$\operatorname{Mat}(\Cobdl^3)$,
$\Kob(k)$, and
$\Kobh(k)$  can be seen
as examples of planar algebras, i.e., if $D$ is a $n$-input planar
diagram, it defines an operation among elements of the previously mentioned
collections. See \cite{Bar1} for specifics of how $D$ defines
operations in each of these collections. In particular, if
$(\Omega_i, d_i)\in\Kob(k_i)$ are complexes,  the complex
$(\Omega,d)=D(\Omega_1,\dots,\Omega_n)$ is defined by
\begin{equation} \label{eq:KobPA}
\begin{split}
    \Omega^r & :=
    \bigoplus_{r=r_1+\dots+r_n} D(\Omega_1^{r_1},\dots,\Omega_n^{r_n})
  \\
    \left.d\right|_{D(\Omega_1^{r_1},\dots,\Omega_n^{r_n})} & :=
    \sum_{i=1}^n (-1)^{\sum_{j<i}r_j}
      D(I_{\Omega_1^{r_1}},\dots,d_i,\dots,I_{\Omega_n^{r_n}}),
\end{split}
\end{equation}
$D(\Omega_1,\dots,\Omega_n)$ is used here as an abbreviation of  $D((\Omega_1,d_1),\dots,(\Omega_n,d_n))$.\\
\indent In \cite{Bar1}
the following very desirable property is also proven. The Khovanov
homology is a planar algebra morphism between the planar algebras
$\mathcal{T}(s)$ of oriented tangles and
$\Kobh(k)$. That is to say,
for an $n$-input planar diagram $D$, and suitable tangles
$T_1,...,T_n$, we have

\begin{equation}\label{eq:planarmorphism}Kh(D(T_1,...,T_n))\cong D(Kh(T_1),...,Kh(T_n)).\end{equation}\\
 \indent  This last property is used in \cite{Bar2} to show a local algorithm for computing the Khovanov homology of a
 link. In that paper, it is explained how it is possible to remove the loops in the smoothings, and some terms in the Khovanov complex $Kh(T_i)$ associated to the local tangles $T_1,...,T_n$, and then combine them together in an $n$-input planar diagram $D$  obtaining
  $D(Kh(T_1),...,Kh(T_n))$, and the Khovanov homology of the original
  tangle.\\
  \indent The elimination of loops and terms can be done thanks to the
  following: Lemma 4.1 and Lemma 4.2 in \cite{Bar2}. We copy these lemmas
  verbatim:
\begin{lemma} \label{lem:Delooping} (Delooping) If an object $S$ in $\Cobdl^3$
contains a closed loop $\ell$, then it is isomorphic (in
$\Mat(\Cobdl^3)$) to the direct sum of two copies $S'\{+1\}$ and
$S'\{-1\}$ of $S$ in which $\ell$ is removed, one taken with a
degree shift of $+1$ and one with a degree shift of $-1$.
Symbolically, this reads
$\bigcirc\equiv\emptyset\{+1\}\oplus\emptyset\{-1\}$.
\end{lemma}
The isomorphisms for the proof can be seen in:
\begin{equation} \label{eq:Delooping}
  \begin{array}{c}
    \includegraphics[height=1.0in]{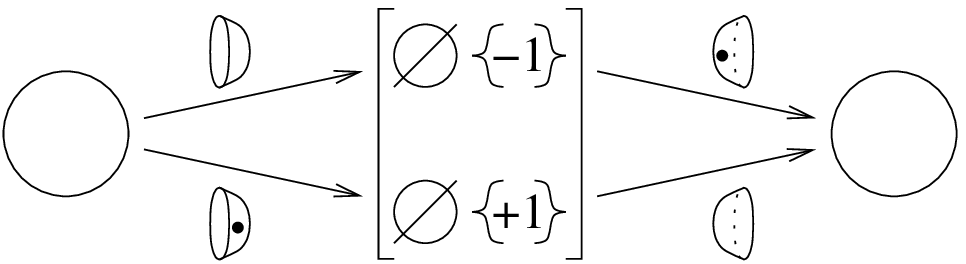}
  \end{array}
\end{equation} using all the relations in (\ref{eq:LocalRelations}).
\begin{lemma} \label{lem:GaussianElimination} (Gaussian elimination, made
abstract) If $\phi:b_1\to b_2$ is an isomorphism (in some additive
category $\calC$), then the four term complex segment in
$\Mat(\calC)$
\begin{equation} \label{eq:BeforeGE}
  \xymatrix@C=2cm{
    \cdots\
    \left[C\right]
    \ar[r]^{\begin{pmatrix}\alpha \\ \beta\end{pmatrix}} &
    {\begin{bmatrix}b_1 \\ D\end{bmatrix}}
    \ar[r]^{\begin{pmatrix}
      \phi & \delta \\ \gamma & \epsilon
    \end{pmatrix}} &
    {\begin{bmatrix}b_2 \\ E\end{bmatrix}}
    \ar[r]^{\begin{pmatrix} \mu & \nu \end{pmatrix}} &
    \left[F\right] \  \cdots
  }
\end{equation}
is isomorphic to the (direct sum) complex segment
\begin{equation} \label{eq:DuringGE}
  \xymatrix@C=3cm{
    \cdots\
    \left[C\right]
    \ar[r]^{\begin{pmatrix}0 \\ \beta\end{pmatrix}} &
    {\begin{bmatrix}b_1 \\ D\end{bmatrix}}
    \ar[r]^{\begin{pmatrix}
      \phi & 0 \\ 0 & \epsilon-\gamma\phi^{-1}\delta
    \end{pmatrix}} &
    {\begin{bmatrix}b_2 \\ E\end{bmatrix}}
    \ar[r]^{\begin{pmatrix} 0 & \nu \end{pmatrix}} &
    \left[F\right] \  \cdots
  }.
\end{equation}
Both these complexes are homotopy equivalent to the (simpler)
complex segment
\begin{equation} \label{eq:AfterGE}
  \xymatrix@C=3cm{
    \cdots\
    \left[C\right]
    \ar[r]^{\left(\beta\right)} &
    {\left[D\right]}
    \ar[r]^{\left(\epsilon-\gamma\phi^{-1}\delta\right)} &
    {\left[E\right]}
    \ar[r]^{\left(\nu\right)} &
    \left[F\right] \  \cdots
  }.
\end{equation}
Here $C$, $D$, $E$ and $F$ are arbitrary columns of objects in
$\calC$ and all Greek letters (other than $\phi$) represent
arbitrary matrices of morphisms in $\calC$ (having the appropriate
dimensions, domains and ranges); all matrices appearing in these
complexes are block-matrices with blocks as specified. $b_1$ and
$b_2$ are billed here as individual objects of $\calC$, but they can
equally well be taken to be columns of objects provided (the
morphism matrix) $\phi$ remains invertible.
\end{lemma}

 \indent From the previous lemmas we infer that the
Khovanov complex of a tangle is homotopy equivalent to a chain
complex without loops in the smoothings, and in which every
differential is a non-invertible cobordism. In other words, if $(\Omega,d)$ is a complex in $\Cobdl^3$, we can use lemmas \ref{lem:Delooping}, \ref{lem:GaussianElimination}, and obtain a homotopy equivalent chain complex $(\Omega',d')$ with no loop in its smoothings and no invertible cobordism in its differentials. We say that a complex that has these two properties is \emph{reduced}. Moreover, we call $(\Omega',d')$   a {\it reduced form} of $(\Omega,d)$ .\\


\section{The category $\Kobo$ and alternating planar algebras}  \label{sec:Alternating}
 In this section we  introduce an alternating orientation in the objects of $\Cobdl^3(k)$. This orientation induces an orientation in the cobordisms of this category. These oriented $k$-strand smoothings and cobordisms form the objects and morphisms in
a new category. The composition between cobordisms in
this oriented category is defined in the standard way, and it is regarded as a
graded category, in the sense of \cite[Section 6]{Bar1}. We mod out the
cobordisms in this oriented category by the relations in (\ref{eq:LocalRelations})
and denote it as $\Cobo(k)$. Now we can follow \cite{Bar1} and define sequentially the categories $\Mat(\Cobo(k))$, $\Kom(\Mat(\Cobo(k)))$ and  $\Komh(\Mat(\Cobo(k)))$. These last two categories are what we denote $\Kobo(k)$, and $\Koboh$. As usual, we use $\Kobo$, and $\Koboh$, to denote $\bigcup_k \Kobo(k)$ and $\bigcup_k \Koboh(k)$ respectively.\\

The orientation of the smoothings is done in such a way that the orientation of the strands is alternating in the boundary of the disc where they are embedded. 
After removing loops, the resulting collection of alternating oriented smoothings obtained will be denoted by the symbol $\calSo$. A $d$-input
planar diagram with an alternating orientation of its arcs  provides a good tool for the
horizontal composition of objects in $\calSo$. Given smoothings
$\sigma_1,...,\sigma_d$, a suitable alternating $d$-input planar diagram $D$ to
compose them has the property that the $i$-th input disc has as many
boundary arc points as $\sigma_i$. Moreover placing $\sigma_i$ in the $i$-th input
disc, the orientation of $\sigma_i$  and $D$ match. An alternatively oriented $d$-input
planar diagram as this  provides a good tool for the
horizontal composition of objects not only in $\calSo$, but also in $\Cobo$, $\Mat(\Cobo(k))$, $\Kobo$, and $\Koboh$.\\

 We are going to use these alternating diagrams to compose non-split alternating tangles, and we want to preserve the non-split property of the tangle. Hence, it will be better if we use $d$-input type $\calA$ diagrams.
A $d$-input type-$\mathcal{A}$ diagram  has an even number of
strings ending in each of its boundary components, and by condition (ii) there is complete connection among input discs of the diagram and its arcs, so every string
that begins in the external boundary ends in a boundary of an
internal disk. We can classify the strings as:
{\it curls}, if they have their ends in the same input disc; {\it interconnecting arcs}, if their ends are in
different input discs, and {\it boundary arcs}, if they have one end in an input disc and the other in the external boundary
of the output disc. The arcs and the boundaries of the discs divide the surface of the diagram into disjoint regions. Diagrams with only one or two input discs deserve special attention. Operators defined from diagrams like these are very important for our purposes since some of them are considered as the generators of the entire collection of operators in an alternating planar algebra.
\begin{figure}[th] \centering
\includegraphics[scale=.6]%
{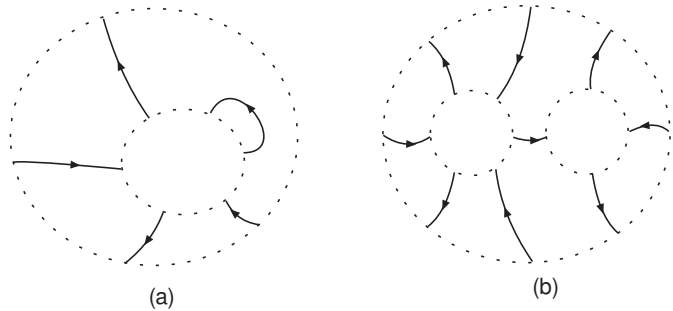}%
\caption{Examples of basic planar diagrams}\label{Fig:bplanar}
\end{figure}
\begin{definition} A basic planar diagram (See Figure \ref{Fig:bplanar}) is a 1-input alternating planar
diagram with a curl in it,  or a 2-input alternating planar diagram with only one interconnecting arc. A basic operator is one defined from a basic planar diagram.
A negative unary basic operator is one defined from a basic 1-input diagram where the curl completes a negative loop. A positive unary basic operator is one defined from a basic 1-input diagram where the curl completes a positive loop. A binary basic operator is one defined from a basic 2-input planar diagram.
\end{definition}

\begin{proposition}\label{prop:compplanar} Any operator $D$ in an alternating oriented planar algebra is the finite composition of basic operators. 
\end{proposition}
\begin{proof}
If $D$ has just one input, it is merely a ``closure operation'' 
obtained by connecting together some pairs of in/out strands in the input 
tangle. So in that case $D$ is a composition of basic unary operators as 
in Figure 2a. If $D$ has more than one input, then remembering that $D$ is 
connected, these inputs may first be connected together using binary 
operators as in Figure 2b, and then one may proceed as in the first part 
of the proof. \qed
\end{proof}

It will be useful to analyze how the rotation number of a smoothing changes when it is embedded in these basic operators.  if $U$ is a positive unary basic operator, and $\sigma \in \calSo$ can be embedded in $U$, then the standard closure of $\sigma$ is the same as the standard closure of $U(\sigma)$. So, $R(U(\sigma))=R(\sigma)$. If instead $U$ is a negative unary basic operator, then either the closure of $U(\sigma)$ has one positive loop less than the closure of $\sigma$,  or the closure of $U(\sigma)$ has one negative loop more than the closure of $\sigma$. Thus in all cases $R(U(\sigma))=R(\sigma)-1$; and finally if $B$ is a binary basic operator, and $\sigma_1, \sigma_2 \in \calSo$ can be embedded in its respective input discs, then is easy to see that $R(B(\sigma_1,\sigma_2))=R(\sigma_1)+R(\sigma_2)-1$. All of this leads to   
\begin{proposition}\label{prop:AsoRotNumber}
For each d-input alternating planar diagram $D$  there exists an integer $R_D$ (its associated rotation number), such as for each $d$-tuple of  oriented smoothings $\sigma_1,...,\sigma_d$ in which $D(\sigma_1,...,\sigma_d)$ makes sense, the rotation
number of $D(\sigma_1,...,\sigma_d)$ is given by:
\begin{equation}\label{eq:AsocRotNumber}
R(D(\sigma_1,...,\sigma_d))=R_D+\sum_{i=1}^dR(\sigma_i)\end{equation} 
\end{proposition}


\section{Diagonal complexes } \label{sec:On-Diagonal}
We begin this section with an example of diagonal complex and then we analyze what occurs when diagonal complexes are embedded in basic planar diagrams

\begin{example}\label{Ex:diagonal} As in \cite{Bar2}, a dotted line represent a dotted curtain, and $\HSaddleSymbol$ represents the saddle $\smoothing \longrightarrow \hsmoothing$. The basepoints are not marked as they make no difference. 
\begin{enumerate}
\item
\[\Omega_1= \begin{diagram}
\node{\begin{tabular}{c}\includegraphics[scale=.5]%
{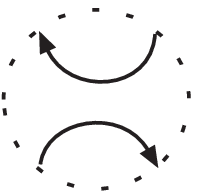}\end{tabular}\{-2\}}\arrow{e,t}{\left[\begin{array}{c}
\includegraphics[scale=.3]%
{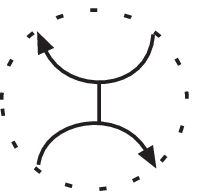} \\
\end{array}\right]}\node{\begin{tabular}{c}\includegraphics[scale=.5]%
{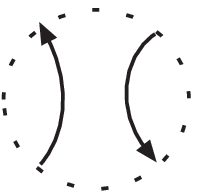}\end{tabular} \{-1\}} \end{diagram}.\] This is the Khovanov homology of the negative crossing $\undercrossing$, now with orientation in the smoothings. Remember that the first term has homological degree -1. In this example the rotation number in the first term is $1$ and in the second term it is $2$. Taking into account the shift in each of these smoothings, the rotation numbers are respectively -1 and 1. Observe that in each case, the difference between 2 times the homological degree $r$ and the shifted rotation number is $-1$.
\item The complex
\[\Omega_2= \begin{diagram}
\node{\begin{tabular}{c}\includegraphics[scale=.5]%
{figs/alternating11.eps}\end{tabular}\{-2\}}\arrow{e,t}{\left[0\right]}\node{\begin{tabular}{c}\includegraphics[scale=.5]%
{figs/alternating11.eps}\end{tabular} \{-1\}} \end{diagram}.\] 
is not diagonal. Observe that in this case $2r - R(\sigma_i\{q_{i}\})$ is not a constant.
\end{enumerate}
\end{example}


\subsection{Applying unary operators}
The reduced complexes in $\Kobo$ can be inserted in appropriate unary planar diagrams. Applying the DG algorithm we obtain again a reduced complex in $\Kobo$. The whole process can be summarized in the following steps:
\begin{enumerate}
\item placing of the complex in the  input disc of the unary planar arc diagram by using equations (\ref{eq:KobPA}) with $n=1$,
\item removing the loops obtained by applying lemma \ref{lem:Delooping}, i.e, replacing each of them by a copy of $\emptyset\{+1\}\oplus\emptyset\{-1\}$, and
\item applying gaussian elimination (lemma \ref{lem:GaussianElimination}), and removing in this way each invertible differential in the complex.
\end{enumerate}
 \begin{definition}Let $(\Omega,d)$ be a chain complex in $\Kobo(k)$, then a partial closure of $(\Omega,d)$ is a chain complex of the form $D_l\circ \cdots \circ D_1(\Omega)$ where $0\leq l<k$ and every $D_i$ ($1\leq i\leq l$) is a unary basic operator.
\end{definition}
We have diagonal complexes whose partial closures are again diagonal complexes.
\begin{definition}\label{def:cohdiagonalcomplex}
Let $(\Omega,d)$ be a $C$-diagonal complex in $\Kobo$ . We say that $(\Omega,d)$ is {\it coherently $C$-diagonal}, if for any unary operator $U$ (where $U(\Omega,d)$ makes sense) with associated rotation number $R_U$, we have that $U(\Omega,d)$ is a ($C-R_U$)-diagonal complex.
\end{definition}

We denote as $\calD(k)$ the collection of all coherently diagonal complexes in $\Kom(\Mat(\Cobo(k)))$, and as usual, we write $\calD$ to denote $\bigcup_k \calD(k)$. It is easy to prove that any coherently diagonal complex satisfies that:

\begin{enumerate}
\item after delooping any of the positive loops obtained in any of its partial closures, by using lemma \ref{lem:GaussianElimination}, the terms with negative shifted-degree can be eliminated.
\item after delooping any of the negative loops obtained in any of its partial closures, by using lemma \ref{lem:GaussianElimination}, the terms with positive shifted-degree can be eliminated.
\end{enumerate}

\begin{example}\label{Ex:cohdiagonal}
Since the computation of any other of its partial closures produces other diagonal complex, the Khovanov homology of  the negative crossing is an element of $\calD(2)$.

\parpic[r]{\includegraphics[scale=.6]%
{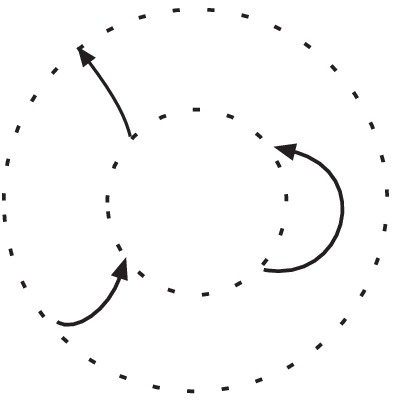}} Indeed, by embedding $\Omega_1$ of the example \ref{Ex:diagonal} in a unary basic planar diagram $U_1$ such as the one on the right which has an associated rotation number $R_{U_1} = 0$, produces the chain complex.

\begin{eqnarray*}U_1(\Omega_1) &=& \begin{diagram}
\node{\left[\begin{tabular}{c}\includegraphics[scale=.5]%
{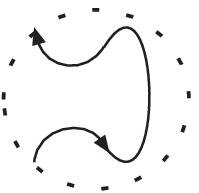}\end{tabular}\{-2\}\right]}\arrow{e,t}{\left[\begin{array}{c}
\includegraphics[scale=.3]%
{figs/saddlealternating12.eps} \\
\end{array}\right]} \node{\left[\begin{tabular}{c}\includegraphics[scale=.5]%
{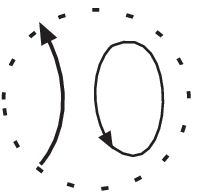}\end{tabular} \{-1\}\right]} \end{diagram}\\
& & \vspace{.5cm}\\
&\sim & \begin{diagram}
\node{\left[\begin{tabular}{c}\includegraphics[scale=.5]%
{figs/alternating121.eps}\end{tabular}\{-2\}\right]}\arrow{e,t}{\left[  \begin{array}{c}
\includegraphics[scale=.3]%
{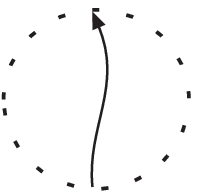} \\
\includegraphics[scale=.3]%
{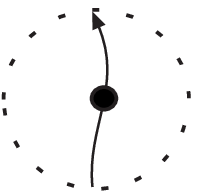} \\
\end{array}
\right]
} \node{\left[\begin{array}{l}\begin{tabular}{c}\includegraphics[scale=.5]%
{figs/alternating131.eps}\end{tabular}\{-2\}\\
\begin{tabular}{c}\includegraphics[scale=.5]%
{figs/alternating131.eps}\end{tabular}\{0\}\end{array}\right]} \end{diagram}
\end{eqnarray*}
The last complex is the result of applying lemma \ref{lem:Delooping}. Now applying lemma \ref{lem:GaussianElimination}, we obtain a homotopy equivalent complex \begin{eqnarray*}U_1(\Omega_1) &\sim& \begin{diagram}
\node{0}\arrow{e,t}{\left[0\right]} \node{\left[\begin{tabular}{c}\includegraphics[scale=.5]%
{figs/alternating131.eps}\end{tabular} \{0\}\right]} \end{diagram}
\end{eqnarray*}
which is also a diagonal complex with the same rotation constant -1. Obviously a complete proof  that this complex is coherently diagonal involves checking that the same happens when embedding the complex in each of its partial closures.\\

\end{example}

\begin{example}
The complex
\[\Omega_3= \begin{diagram}
\node{\begin{tabular}{c}\includegraphics[scale=.5]%
{figs/alternating12.eps}\end{tabular}\{-2\}}\arrow{e,t}{\left[0\right]}\node{\begin{tabular}{c}\includegraphics[scale=.5]%
{figs/alternating11.eps}\end{tabular} \{-1\}} \end{diagram}.\] 
is diagonal but it is not coherently diagonal. Indeed, if we embed it in $U_1$ the result is
 \begin{eqnarray*}U_1(\Omega_3) &=& \begin{diagram}
\node{\left[\begin{tabular}{c}\includegraphics[scale=.5]%
{figs/alternating121.eps}\end{tabular}\{-2\}\right]}\arrow{e,t}{\left[ 0 \right]
} \node{\left[\begin{array}{l}\begin{tabular}{c}\includegraphics[scale=.5]%
{figs/alternating131.eps}\end{tabular}\{-2\}\\
\begin{tabular}{c}\includegraphics[scale=.5]%
{figs/alternating131.eps}\end{tabular}\{0\}\end{array}\right]} \end{diagram}
\end{eqnarray*}
which is not diagonal.
\end{example}

\subsection{Applying binary operators}

\begin{proposition}\label{prop:ComplexUnaryToBinary} Let $\Omega_1$ and $\Omega_2$ be coherently diagonal complexes, and let $D$ be a binary basic planar operator for which $D(\Omega_1,\Omega_2)$ is well defined. For each partial closure $C(D(\Omega_1,\Omega_2))$, there exists an operator $D'$ defined using a diagram without curls and reduced chain complexes $\Omega_1',\Omega_2'$ in $\calD$ such that \[C(D(\Omega_1,\Omega_2))=D'(\Omega_1',\Omega_2')\]. \end{proposition}

\parpic[r]{\includegraphics[scale=.40]%
{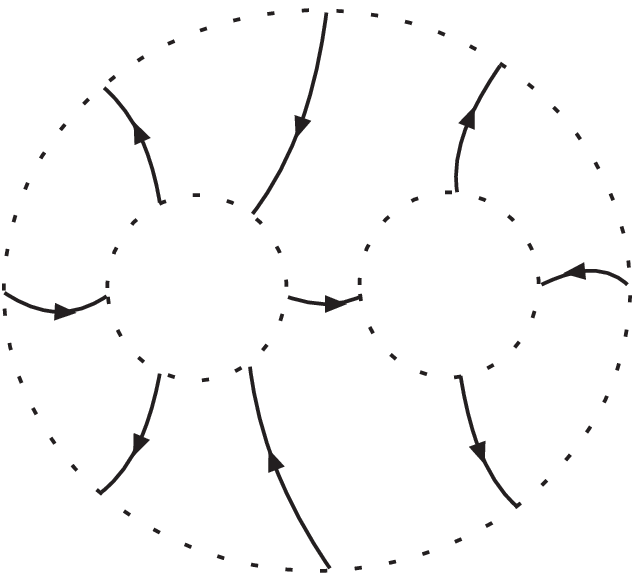}}
\begin{proof} \noindent We have a binary basic planar diagram $D$ such as the one at the right. A closure of $D(\Omega_1,\Omega_2)$ can be regarded as the composition of $\Omega_1$ and $\Omega_2$ in an operator $C(D)$ defined from this closure. We can also regard the disc $C(D)$ as a composition of two closure discs $E,E'$ embedded in a binary planar diagram $D'$ with no curl such that $C(D)=D(E,E')$. See Figure \ref{Fig:iqualplanars}.
 Hence, $C(D(\Omega_1,\Omega_2))=D(E(\Omega_1),E'(\Omega_2))$. Since $E(\Omega_1)$ and $E'(\Omega_2)$ are respectively closures of $\Omega_1$ and $\Omega_2$ which are elements of $\calD$, the proposition has been proved. \qed
\end{proof}

\begin{figure}[hbt] \centering

\includegraphics[scale=.5]%
{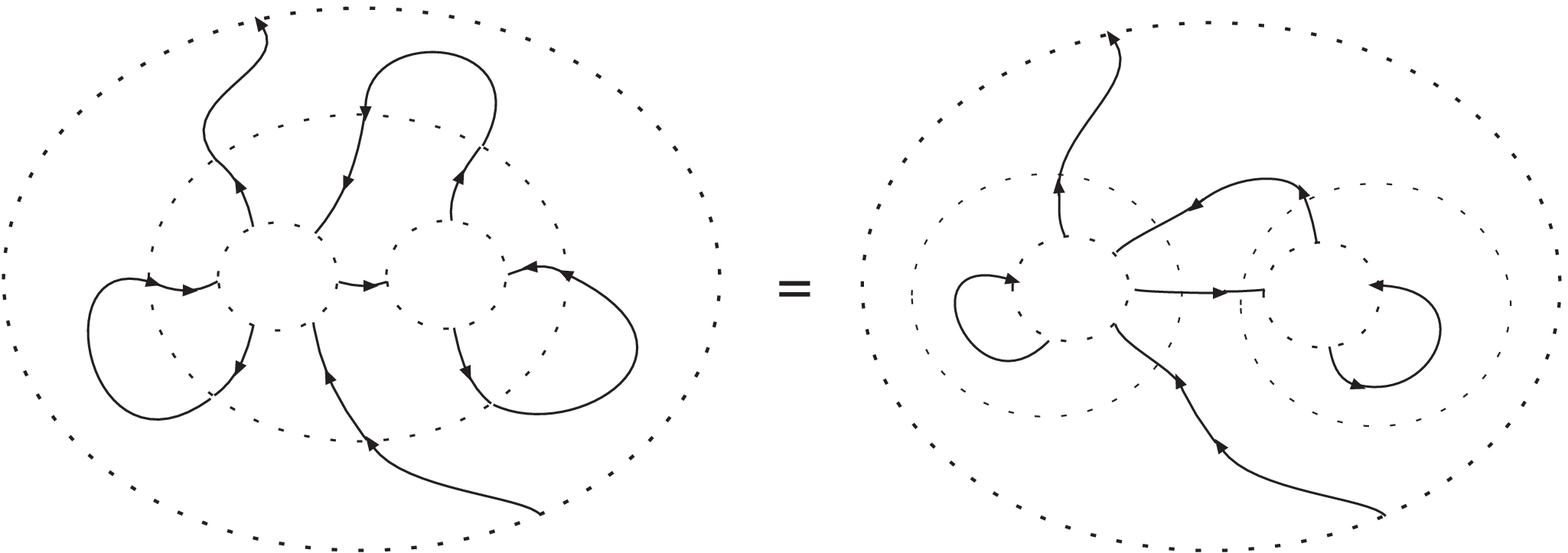}%
\caption{The closure of a binary operator can be considered as the binary composition of two unary operator closures}\label{Fig:iqualplanars}
\end{figure}


\section{Perturbed double Complexes}\label{URC}
The section that follows can be reformulated using the language of homological perturbation theory (e.g. \cite{Cr}, with his $(b,\delta)$ replaced by our soon-to-be-introduced $(d^0,\sum_{i \geq 1}d^i)$). Yet given the relative inaccessibility of the required "homological perturbation lemma," we have chosen formulate only what we need, using the language of the rest of this paper.

Given an additive category $\calC$, an \emph{(upward) perturbed double complex} in $\calC$ is a family $\Omega$ of objects $\{\Omega_{p,q}\}$ of $\calC$ indexed in $\mathbb{Z} \times \mathbb{Z}$, together with morphisms \[d^i:\Omega_{p,q}\rightarrow \Omega_{p-i+1,q+i} \, \text{ for each } \, i \geq 0,\]

such that if $\mathbf{d}=\sum d^i$ then $\mathbf{d}^2=0$; or alternatively,
\begin{equation}\label{Eq:reticularcomplexcondition}\sum_{i=0}^k d^i\circ d^{k-i}=0  \, \text{ for each } \, k \geq 0 \end{equation}

It will be convenient to illustrate the perturbed double complex as a lattice  as in Figure~\ref{Fig:PerturbedComplex} in which any node $\Omega_{p,q}$ is the domain of  arrows $d^0,d^1,\ldots$, which 
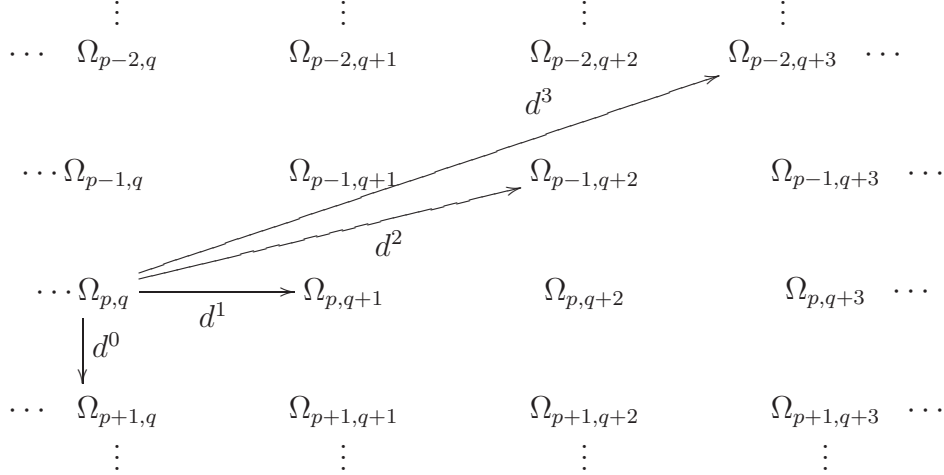
\begin{figure}[hbt] \centering
\begin{equation*} 
  \xymatrix@C=1.5cm{
{\begin{matrix}  & \vdots \\ \cdots & \Omega_{p-2,q} \end{matrix}}  & {\begin{matrix}  \vdots \\ \Omega_{p-2,q+1} \end{matrix}}
  & {\begin{matrix}  \vdots \\ \Omega_{p-2,q+2} \end{matrix}} &  *+[l]{\begin{matrix} \vdots & \\  \Omega_{p-2,q+3}& \cdots \end{matrix}} \\
    \cdots \Omega_{p-1,q} & \Omega_{p-1,q+1} & \Omega_{p-1,q+2}  &  {\begin{matrix} \Omega_{p-1,q+3} & \cdots \end{matrix}} \\
        \cdots \Omega_{p,q} \ar[d]^(.4){\begin{matrix}d^0\end{matrix}} \ar[r]_{\begin{matrix}d^1\end{matrix}} \ar[urr] _(.6){\begin{matrix}d^2\end{matrix}} \ar[uurrr]^(.6){\begin{matrix}d^3\end{matrix} } &  \Omega_{p,q+1}  &   \Omega_{p,q+2}  &  {\begin{matrix} \Omega_{p,q+3}  & \cdots \end{matrix}} \\
 {\begin{matrix}  \cdots & \Omega_{p+1,q} \\  & \vdots \end{matrix}}  & {\begin{matrix}  \Omega_{p+1,q+1} \\ \vdots \end{matrix}}  &  {\begin{matrix} \Omega_{p+1,q+2} \\ \vdots \end{matrix}} & {\begin{matrix}  \Omega_{p+1,q+3} & \cdots \\ \vdots & \end{matrix}}   }
\end{equation*}
\caption{A perturbed double complex}\label{Fig:PerturbedComplex}
\end{figure}
 satisfies the following infinite number of conditions
\begin{description}
\item[For $k=0$] Equation (\ref{Eq:reticularcomplexcondition}) reduces to $d^0\circ d^0 = 0 $.  This condition is equivalent to saying that for each fixed $q \in \mathbb{Z}$,  the objects $\Omega_{p,q}$ and the morphisms $d^0:\Omega_{p,q} \rightarrow \Omega_{p+1,q}$ form a complex. We call these complexes the \emph{vertical complexes} $\Omega_{\bullet,q}$ of the reticular complex

\item[For $k=1$] Equation (\ref{Eq:reticularcomplexcondition}) reduces to $d^0\circ d^1 + d^1\circ d^0= 0 $. This condition is equivalent to saying that all the squares in the diagram anticommute. 
\begin{equation*} 
  \xymatrix@C=1cm{
& \vdots \ar[d]^{d^0}   & \vdots \ar[d]^{d^0}   & \\
    \cdots   \ar[r]^{d^1} & \bullet \ar[r]^{d^1} \ar[d]^{d^0} &   \bullet \ar[r]^{d^1} \ar[d]^{d^0} &  \cdots \\
    \cdots   \ar[r]^{d^1} &  \bullet \ar[r]^{d^1} \ar[d]^{d^0} &   \bullet \ar[r]^{d^1} \ar[d]^{d^0} &  \cdots \\
    & \vdots   & \vdots    & }
\end{equation*}
\item[For $k=2$] Equation (\ref{Eq:reticularcomplexcondition}) reduces to $d^0\circ d^2 +d^1\circ d^1 + d^0\circ d^2= 0 $. This states that for each $p,q$ the sum of $d^1\circ d^1$ plus the compositions along consecutive sides of the parallelogram with vertices on $\Omega_{p,q}$, $\Omega_{p-1,q-2} $, $\Omega_{p+1,q} $ and $\Omega_{p,q+2} $  is zero.
\begin{equation*} 
  \xymatrix@C=1cm{
 \bullet    & \bullet   & \bullet     \ar[d]^{d^0} \\
 \bullet  \ar[urr]^{d^2}  \ar[r]_{d^1} \ar[d]_{d^0} &   \bullet \ar[r]^{d^1}  &  \bullet  \\
  \bullet \ar[urr]_{d^2}   & \bullet  &  \bullet \\
 }
\end{equation*}
\item[ For any $k \geq 0$] Equation (\ref{Eq:reticularcomplexcondition})  states that the sum of the compositions along consecutive sides of all possible \emph{parallelograms}  with diagonal on $\Omega_{p,q}$ and $\Omega_{p-k,q+k+2} $ is zero. 
\begin{equation*} 
  \xymatrix@C=1cm{
 \bullet    & \bullet   & \bullet   & \bullet \ar[d]^{d^0}   \\
 \bullet   &   \bullet   &  \bullet \ar[r]^{d^1}   & \bullet   \\
 \bullet \ar[d]_{d^0}  \ar[r]_{d^1} \ar[urr]_{d^2}  \ar[uurrr]^{d^3}  &   \bullet \ar[urr]^{d^2}  &   \bullet  & \bullet   \\
 \bullet  \ar[uurrr]_{d^3}   &   \bullet  &   \bullet  & \bullet   \\ }
\end{equation*}
\end{description}
We must include in the sum, a composition $d^{\sfrac{k}{2}}\circ d^{\sfrac{k}{2}}$ along the common diagonal of the parallelograms, for cases where $k$ is an even integer. \\

In the same way as in the case of double complexes, a perturbed double complex is associated to a chain complex that we denominate its \emph{total complex}, abbreviated $\operatorname{Tot}(\Omega)$, and defined as follows:
\begin{definition} Given a perturbed double complex $\Omega$, its total complex $\operatorname{Tot}(\Omega)$ is defined by
\[\operatorname{Tot}(\Omega)^n=\bigoplus_{p+q=n}\Omega_{p,q}\]

\[\mathbf{d}\big|_{\Omega_{p,q}}=\sum_{i\geq 0}d^i\]
\end{definition}
 Note that the condition stated by equation (\ref{Eq:reticularcomplexcondition}) makes certain that $\operatorname{Tot}(\Omega)$ is indeed a chain complex. We observe also that double complexes are just the special cases of  perturbed double complexes in which $d^i=0$ for each $i \geq 2$. \\
 \indent If no confusion arises, from now on we omit specific mention of the adjective total and we will write just $\Omega$ when we refer to $\operatorname{Tot}(\Omega)$. We shall simply say ``perturbed double complex" to mean the total complex associated to it.\\
 \indent One desired feature of perturbed double complex $\Omega$ is that the DG algorithm works well when applied to one of its vertical complexes $\Omega_{\bullet, q}$. What we mean with this last sentence is that  the homotopy equivalent complex obtained after applying the DG algorithm in objects and morphisms located in the same vertical complex of a perturbed double complex  is itself a perturbed double complex. We see this inmmediately.\\
 \indent First of all, by applying Lemma \ref{lem:Delooping} in $\Omega_{p,q}$, we do not change the configuration of perturbed double complex. Indeed, if $f:\Omega_{p,q} \rightarrow \Omega'_{p,q}$ is an isomorphism, then it is possible to obtain a perturbed double complex $\Omega'$ homotopy equivalent to $\Omega$ by substituting $\Omega_{p,q}$ by $\Omega'_{p,q}$, and by replacing any morphism $d^i$ with image in $\Omega_{p,q}$ by the morphism $f\circ d^i$, and any morphism $d^j$ with domain in $\Omega_{p,q}$ by $d^j\circ f^{-1}$. \\
\indent Secondly, if $\phi:b_1 \to b_2$ is an isomorphism in $\calC$, and if $D_0, E_0$ are column vectors of object in $ \calC$. Given  a perturbed double complex $\Omega$ with  
 \[ \Omega_{p,q}=\begin{bmatrix}b_1 \\ D_0\end{bmatrix} \text{ and } \,\, \Omega_{p+1,q}=\begin{bmatrix}b_2 \\ E_0\end{bmatrix},\]
 then eliminating $\phi$ by applying Lemma \ref{lem:GaussianElimination} does not bring any change in a vertical chain $\Omega_{\bullet,r}$ with $r\neq q$. Moreover, since the application of this lemma does not bring any new type of arrow in $\Omega$, we have that the homotopy equivalent complex obtained is also a perturbed double complex. Indeed, after applying Lemma \ref{lem:GaussianElimination}, the arrows \[d^j=\begin{pmatrix}\alpha_j \\ \beta_j \end{pmatrix}:\Omega_{p+j-1,q-j} \to   \Omega_{p,q}, \hspace{.5cm} d^i=\begin{pmatrix}\gamma_i \,\,\, \epsilon_{i0}\end{pmatrix}:\Omega_{p,q} \to   \Omega_{p-i+1,q+i}, \]
 \[d^j=\begin{pmatrix}\delta_j \\ \epsilon_{0j} \end{pmatrix}:\Omega_{p+j,q-j} \to   \Omega_{p+1,q}, \hspace{.5cm} d^i=\begin{pmatrix}\mu_i \,\,\, \nu_{i}\end{pmatrix}:\Omega_{p+1,q} \to   \Omega_{p-i+2,q+i},  \]
 and
 \[ d^{i+j}:\Omega_{p+j,q-j} \to \Omega_{p-i+1,q+i} \]
 have been change to
 \[d^j=\begin{pmatrix}\beta_j \end{pmatrix}:\Omega_{p+j-1,q-j} \to   \Omega'_{p,q}, \hspace{.5cm} d^i=\begin{pmatrix} \epsilon_{i0}-\gamma_i\phi^{-1}\delta_0 \end{pmatrix}:\Omega'_{p,q} \to   \Omega_{p-i+1,q+i}, \]
 \[d^j=\begin{pmatrix} \epsilon_{0j}-\gamma_0\phi^{-1}\delta_j \end{pmatrix}:\Omega_{p+j,q-j} \to   \Omega'_{p+1,q}, \hspace{.5cm} d^i=\begin{pmatrix} \nu_{i}\end{pmatrix}:\Omega'_{p+1,q} \to   \Omega_{p-i+2,q+i},  \]
 and
 \[ d^{i+j}-\gamma_i\phi^{-1}\delta_{j}:\Omega_{p+j,q-j} \to \Omega_{p-i+1,q+i} \]
  where $\delta_0$ and $\gamma_0$ are the morphisms $\delta_0:D_0 \to b_2$  and $\gamma_0:b_1 \to E_0$.\\
 \indent A consequence of all of this is that the DG algorithm can be applied to a vertical complex in $\Omega$ in such a way that  the others vertical complexes remain unchanged.


\section{Proof of Theorem \ref{Theorem:MainTheo2}}\label{sec:Theorem2} 
The main part of the proof of Theorem  \ref{Theorem:MainTheo2} is to show that the composition of coherently diagonal complexes in a binary basic operator is also coherently diagonal. So, before proving this theorem, let us analyze first what occurs when in this type of operator two smoothings are embedded. Recall that $\calSo$ denotes the class of alternating oriented smoothings.
\begin{proposition}\label{prop:ComplexSmoothingDiag}Let $\sigma$ and $\tau$ be smoothings in $\calSo$, and let $D$ be a suitable binary planar operator defined from a no-curl planar arc diagram with output disc $D_0$, input discs $D_1,D_2$, associated rotation constant $R_D$ and with at least one boundary arc ending in $D_1$. Then there exists a closure operator $E$ and a unary operator $U$ such that $D(\sigma, \tau)=U(E(\sigma))$. Moreover, if $(\Omega,d)$ is coherently $C$-diagonal, then $D(\Omega,\tau)$ is ($C-R(\tau)-R_D$)-diagonal.
\end{proposition}
  
\parpic[r]{\includegraphics[scale=.60]%
{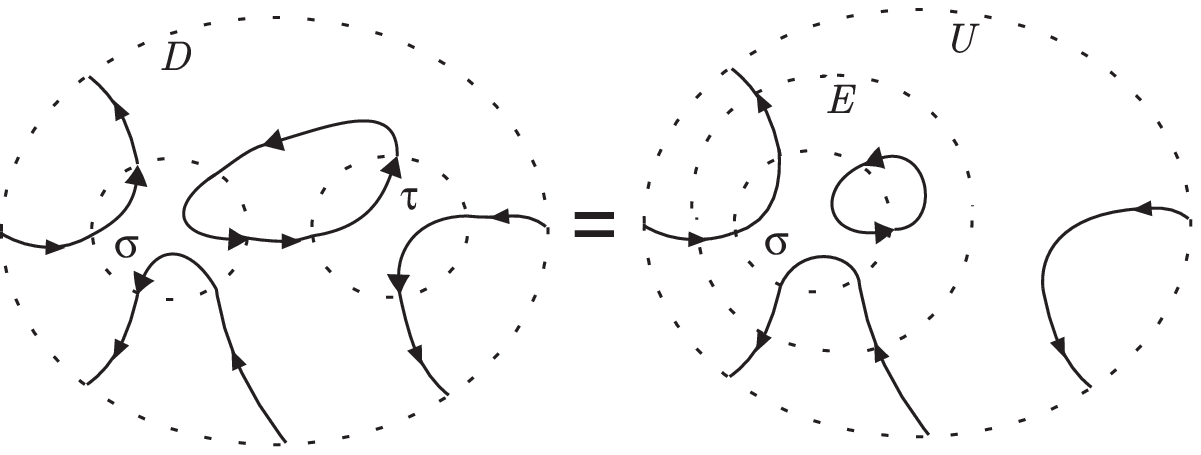}}
\begin{proof}  The picture on the right displays the equivalence between $D(\sigma, \tau)$ and $U(E(\sigma))$.  If instead of the smoothing $\sigma$, the coherently diagonal complex $\Omega$ is embedded in the first input disc of $D$, we have that $D(\Omega, \tau)=U(E(\Omega))=U(\Omega')$, where $\Omega'$ and hence $U(\Omega')$ are reduced diagonal complex.  To prove that the rotation constant of $D(\Omega,\tau)$ is $C-R(\tau)-R_D$, we observe that for each smoothing $\sigma\{q_{\sigma}\}$ in $\Omega$ the shifted rotation number satisfies $R(D(\sigma\{q_{\sigma}\},\tau))= R_D+R(\sigma\{q_{\sigma}\})+R(\tau)= R_D+2r -C+R(\tau)$. Therefore, $2r-R(D(\sigma\{q_{\sigma}\},\tau))=C-R(\tau)-R_D$ \qed
\end{proof}

\begin{proposition}\label{prop:InputComplexVector} Let $\Omega$  be a coherently $C$-diagonal complex. Let $\left[\sigma_j\right]_j$ be a vector of degree-shifted smoothings in $\calSo$, all of them with the same rotation number $R$.  Suppose that $D$ is an appropriate binary operator defined from a no-curl planar arc diagram with  associated rotation constant $R_D$ and at least one boundary arc coming from the first input disc.  Then $D(\Omega,\left[\sigma_j\right]_j)$ is a $(C-R-R_D)$-diagonal complex.
\end{proposition}
\begin{proof}
The complex $\Omega$ is homotopy equivalent to a reduced $C$-diagonal complex $\Omega'$ .The complex $D(\Omega',\left[\sigma_j\right]_j)$ is the direct sum $\oplus_j \left[ D\left(\Omega', \sigma_j\right)\right]$. Thus, the proposition follows from the observation that by  proposition \ref{prop:ComplexSmoothingDiag}, each of its direct summands $D\left(\Omega', \sigma_j\right)$ is a diagonal complex with rotation constant $C-R-R_D$. 
\qed
\end{proof}

\begin{lemma}\label{Lemma:2InputDiagonalComplex} Let $\Omega_1$  be a coherently $C_1$-diagonal complex. Let $\Omega_2$ be a $C_{2}$-diagonal complex.  Suppose that $D$ is an appropriate binary operator defined from a no-curl planar arc diagram with  associated rotation constant $R_D$ and at least one boundary arc coming from the first input disc.  Then $D(\Omega_1,\Omega_2)$ is $(C_{1}+C_{2}-R_D)$-diagonal.
\end{lemma}

\begin{proof}
Observe that $\Omega=D(\Omega_1,\Omega_2)$ is a double complex. Indeed, if $\Omega_2$ is the chain complex \[ \cdots \longrightarrow \Omega_2^{q-1} \longrightarrow \Omega_2^q \longrightarrow \Omega_2^{q+1} \dots \] then  $\Omega_{\bullet, q}$ is the planar composition $D(\Omega_1,\Omega_2^q)$. Assume that $\Omega_2$ is in its reduced form, then any of the smoothings in $\Omega_2^q$ has the same rotation number, $2q-C_2$. Thus, by proposition \ref{prop:InputComplexVector}, $\Omega_{\bullet, q}$ is homotopy equivalent to a reduced diagonal complex $\Omega'_{\bullet, q}$ with rotation constant  $C_{1}+C_2-2q-R_D$.  We already know that we can  apply delooping and gaussian elimination in $\Omega$ involving only elements of $\Omega_{\bullet, q}$ and obtain a homotopy equivalent complex that has no changes in another vertical chain complex of $\Omega$. In consequence, $\Omega$ is homotopy equivalent to a perturbed complex  $\Omega'$ in which each $\Omega_{\bullet,q}$ has been replaced by its correspondent reduced complex $\Omega'_{\bullet,q}$. Thus, for each obtained $\Omega'_{\bullet,q}$ and each of its homological degree $p$, we have $2p-R(\Omega_{p,q})=C_{1}+C_2-2q-R_D$. Therefore,  $\Omega'$ is a diagonal complex with rotation constant $C_{1}+C_2-R_D$.
\qed
\end{proof}
\vspace{.5cm}
\begin{proof} (Of \textbf{Theorem \ref{Theorem:MainTheo2}}) By proposition \ref{prop:compplanar}, we only need to prove that $\calD$ is closed under composition of basic operators. Let $(\Omega_1,d_1)\in \calD$ and let $U$ be a basic unary operator. Since $U(\Omega_1)$ is a partial closure of $(\Omega_1,d_1)$, $U(\Omega_1)$ is diagonal. Furthermore any partial closure of $U(\Omega_1)$ is also a partial closure of $(\Omega_1,d_1)$, so $U(\Omega_1)\in \calD$.\\

Let $(\Omega_1,d_1)$ and $(\Omega_2,d_2)$ be elements of $\calD$, and let $D$ be  a basic binary operator. By Lemma \ref{Lemma:2InputDiagonalComplex}, $D(\Omega_1,\Omega_2)$ is a diagonal complex. Let $C(D(\Omega_1,\Omega_2))$ be a partial closure of $D(\Omega_1,\Omega_2)$. The fact that $C(D)$ is only a partial closure implies that there is at least one boundary arc in $C(D)$. Without loss of generality, we can assume that there is one boundary arc ending in the first input disc of $C(D)$.  By proposition \ref{prop:ComplexUnaryToBinary} there exist $\Omega_1',\Omega_2'\in \calD$ and a binary operator $D'$ defined from a no-curl planar diagram such that $C(D(\Omega_1,\Omega_2))=D'(\Omega_1',\Omega_2')$. By using Lemma \ref{Lemma:2InputDiagonalComplex}, we obtain that $D'(\Omega_1',\Omega_2')$ is a diagonal complex.  \qed
\end{proof}
\section{Non-split alternating tangles and Lee's theorem} \label{sec:LeeTheorem}
\subsection{Gravity information} \label{subsec:GravityInformation}
Given a diagram of an alternating tangle, we add to it some special
information which will help us to compose the Khovanov invariant
of an alternating tangle in an alternating planar diagram. This is illustrated by drawing, in every
strand of the diagram, an arrow pointing in to the undercrossing, or
equivalently (if we have alternation), pointing out from the overcrossing. In a
neighbourhood of a crossing, the diagram looks like the one in Figure
\ref{gravicross}(a).
\begin{figure}[hbt]
\begin{tabular}{ccc}
\includegraphics[scale=0.6]%
{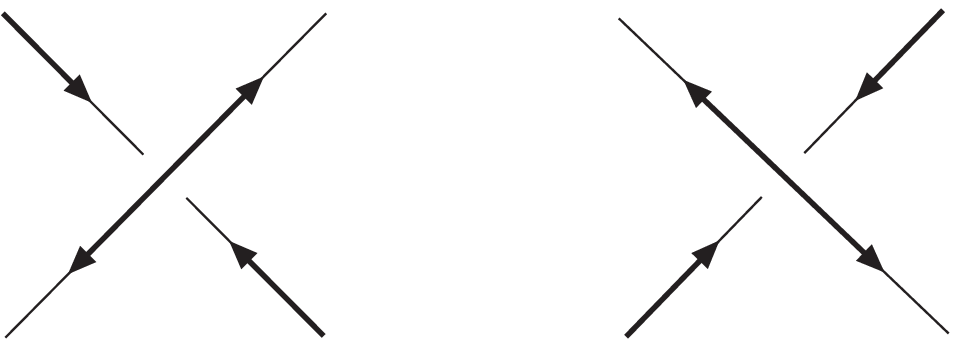} & \hspace{2cm} &\includegraphics[scale=.8]%
{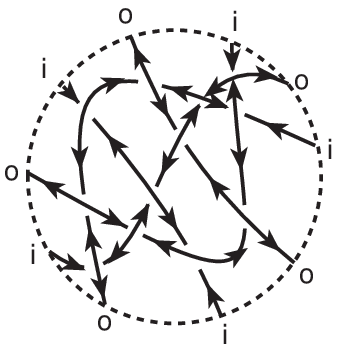} \\
(a) & & (b)
\end{tabular}
\caption{(a) Gravity information in a neighbourhood of a crossing. (b) Gravity information in the tangle. We use: o for out-boundary points
and i for in-boundary points.
}\label{gravicross}
\end{figure}
Figure \ref{gravicross}(b) shows a diagram of a tangle in which we have
added the gravity information to the whole tangle. We observe, (see Figure \ref{gravicross}(a)) that if we make a
smoothing in the crossing, the orientation provided by the gravity
information is preserved, and that a $0$-smoothing is clockwise and
$1$-smoothing is counterclockwise, see figure \ref{Fig:smoothing}.
It is easily observed as well that if we go into a non-split
alternating tangle for an in-boundary point and turn to the right (a
$0$-smoothing) every time that we meet a crossing, we are going to
get out of the tangle along the boundary point immediately to the right. Hence, the in-
and out-boundary points of the diagram of the tangle are arranged
alternatingly. These two observations are stated in the following
two propositions:
\begin{figure}[hbt] \centering

\includegraphics[scale=.8]%
{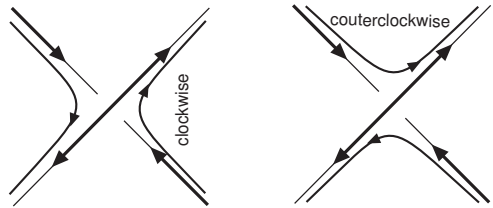}%
\caption{The smoothings in the diagrams preserve the gravity
information}\label{Fig:smoothing}
\end{figure}
\begin{proposition}\label{prop:gravit.preserv}The 0-smoothings and 1-smoothings preserve the gravity information. The first ones provides a clockwise orientation of the pair of strands in the smoothing, and the
last provides a counterclockwise orientation. \qed \end{proposition}
\begin{proposition}\label{prop:AlternanSmoothly}In any non-split alternating tangle, if the $k$-th boundary point is an in-boundary point, then the $(k+1)$-th boundary point is an out-boundary point. \qed \end{proposition}
Propositions \ref{prop:gravit.preserv} and \ref{prop:AlternanSmoothly} indicate that
the smoothings of a tangle could be drawn as trivial tangles in
which arcs are oriented alternatingly. Therefore, the Khovanov homology produces an alternating planar algebra morphism in the sense of \cite{Bar1}.
\subsection{Proof of Theorem  \ref{Theorem:MainTheo1}}
This proof is a direct application of Theorem  \ref{Theorem:MainTheo2}, the fact that the Khovanov homology is an  alternating planar algebras morphism, and the following proposition
\begin{proposition}\label{prop:1-crossing} The Khovanov complex
of a 1-crossing tangle is  coherently diagonal; namely, it is an element of $\calD(2)$.\end{proposition}
\begin{proof} We just need to check each of the possible partial closures of the two 1-crossing tangles to observe that all of them have a reduced diagonal form. One of this can be seen in example \ref{Ex:cohdiagonal}. \qed
\end{proof}

\begin{proof} (Of \textbf{Theorem \ref{Theorem:MainTheo1}})  Any non-split alternating $k$-strand tangle
with $n$ crossings $T$, is obtained by a composition of $n$ 
1-crossing tangles, $T_1$,...,$T_n$, in a $n$-input
type-$\mathcal{A}$ planar diagram. Since the Khovanov homology is a planar algebra morphism,
by using the same $n$-input planar diagram for composing
$\Kh(T_1)$,...,$\Kh(T_n )$ we obtain the Khovanov homology of the original tangle.  According to Theorem \ref{Theorem:MainTheo2}, this is a complex in $\calD$.  \qed \end{proof}
\begin{corollary}\label{cor:1StrandTangleK} The Khovanov complex $[T]$ of a a non-split alternating $1$-tangle $T$ is homotopy equivalent to a complex
\[ \cdots \longrightarrow \Omega^{r}\{2r+K\}\longrightarrow \Omega^{r+1}\{2(r+1)+K\}\longrightarrow\cdots \]
where every $\Omega^r$ is a vector of single lines,
and $K$ is a constant.
\end{corollary}
\begin{proof}
We only have to apply Theorem \ref{Theorem:MainTheo1} and see that the rotation number of a single arc, which is the only simple possible smoothing resulting from a $1$-tangle, is one. \qed
\end{proof}
Figure \ref{Fig:1Diagonal} shows a diagonal complex whose smoothings  have only one strand. Since the rotation number of a smoothing with a unique strand is always 1, we have that the degree shift and the homological degree multiplied by two are in a single diagonal, i.e., $2r- q_r$ is a constant.
\begin{figure}[hbt] \centering
\includegraphics[scale=0.6]%
{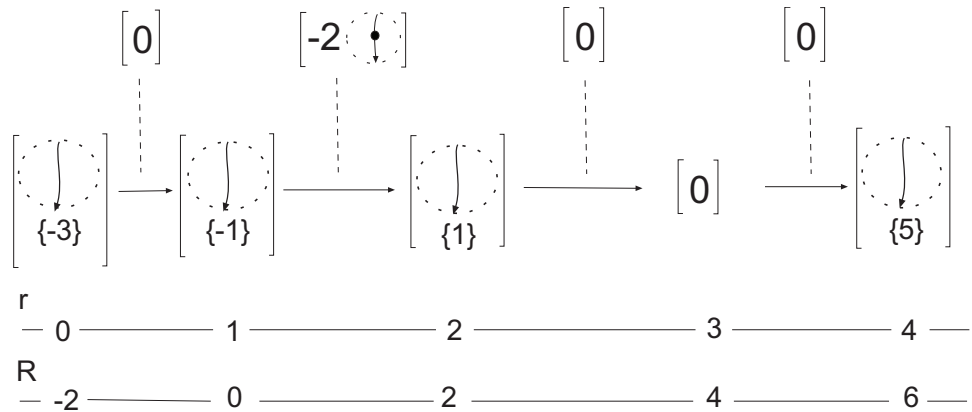}
\caption{A diagonal complex with only one strand in each of its smoothings.
}\label{Fig:1Diagonal}
\end{figure}
\begin{corollary}\label{cor:LeeTheorem}(Lee's theorem)The Khovanov complex $Kh(L)$ of a non-split alternating Link $L$ is homotopy equivalent to a
complex:
\[\cdots \longrightarrow\left(%
\begin{array}{c}
  \Phi^{r}\{q_r+1\} \\
  \Phi^{r}\{q_r-1\} \\
\end{array}%
\right)\longrightarrow \left(%
\begin{array}{c}
  \Phi^{r+1}\{q_{r+1}+1\} \\
  \Phi^{r+1}\{q_{r+1}-1\} \\
\end{array}%
\right)\longrightarrow\cdots \] where every $\Phi^{r}$ is a matrix of
empty 1-manifolds, $q_r=2r+K$, K a constant, and
 every differential is a matrix in the ground ring.
\end{corollary}

\begin{proof}
Every non-split alternating link $L$ is obtained by putting a
1-strand tangle $T$ in a 1-input planar diagram with no boundary. Hence,  by applying the operator defined from this 1-input planar diagrams to the Khovanov complex of this
1-strand tangle, we obtain the Khovanov complex of
a link $L$. By doing that, the vectors of open arcs that we have in corollary \ref{cor:1StrandTangleK} become vectors of circles. Moreover, every cobordism of the complex transforms in a multiple of a dotted
cylinder. Thus, using Lemma \ref{lem:Delooping} converts every
single loop in a pair of empty sets
$ \emptyset\{2r+K+1\}$, $ \emptyset\{2r+K-1\} $ and
every dotted cylinder in an element of the ground ring. \qed
\end{proof}
Figure \ref{Fig:0Diagonal} displays the closure of the complex in Figure \ref{Fig:1Diagonal}. After applying lemmas \ref{lem:Delooping} and \ref{lem:GaussianElimination} we obtain the complex supported in two lines displayed on Figure \ref{Fig:0Diagonal}, as stated in  Lee's Theorem.
\begin{figure}[hbt] \centering
\includegraphics[scale=0.6]%
{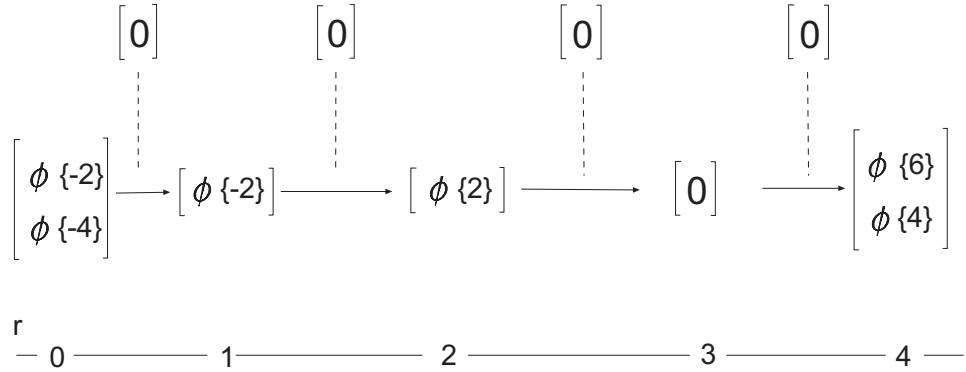}
\caption{A closure of a coherently diagonal complex is a width-two complex.
}\label{Fig:0Diagonal}
\end{figure}
\begin{remark}
It is clear that if our ground ring is $\mathbb{Q}$, as in the case
of \cite{Lee1}, we can use repeatedly lemma
\ref{lem:GaussianElimination} in the complex in corollary \ref{cor:LeeTheorem} and obtain a complex 
whose differential is zero, i.e, the Khovanov homology of the
link.
\end{remark}

\end{document}